\documentclass[12pt]{article}

\newif\ifskiptext\skiptextfalse

\makeatletter%shorter space after section nr.
\def\@seccntformat#1{\csname the#1\endcsname.\hspace*{0.5em}}

\def\@maketitle{%remove date
  \newpage
  \null
  \vskip 2em%
  \begin{center}%
  \let \footnote \thanks
    {\LARGE \@title \par}%
    \vskip 1.5em%
    {\large
      \lineskip .5em%
      \begin{tabular}[t]{c}%
        \@author
      \end{tabular}\par}%
  \end{center}%
  \par
  \vskip 1.5em}%\@maketitle
\makeatother

\usepackage{mathrsfs}
\usepackage{amsmath}
\usepackage{amssymb}
\usepackage{graphics}
\usepackage{url}
\usepackage{enumerate}

\usepackage{xspace}

\usepackage{mynatbib}
\def\citet{\citep} %%% no "text" citations in crossrefs in biblio

\def\citeP{\citep*} %%%ek3 * ==> full author list
\bibpunct[, ]{[}{]}{;}{a}{}{,}
\newcommand{\EKhref}[2]{URL: \url{#1}}
\usepackage[%
breaklinks%
,colorlinks%
,linkcolor=red% internal document links
,anchorcolor=blue%
,pagecolor=blue%
,citecolor=blue%
,bookmarks=false%
]{hyperref}
\usepackage{breakurl}

\usepackage{theorem}
\newtheorem{theorem}{Theorem}[section]

\newtheorem{prop}[theorem]{Proposition}
\newtheorem{cor}[theorem]{Corollary}
\newtheorem{lemma}[theorem]{Lemma}
\newtheorem{algorithm}%ek2 remove [theorem]
{Algorithm}[section]%ek2 add
{\theorembodyfont{\upshape} \newtheorem{remark}%ek2 remove [theorem]
{Remark}[section]}
\newtheorem{define}%ek2 remove [theorem]
{Definition}
{\theorembodyfont{\upshape} \newtheorem{example}%ek2 remove [theorem]
{Example}[section]%ek2 add
}
{Problem}[section]%ek2 add

\def\citet{\cite}

\newcommand{\comment}[1]{}
\def\cal{\mathcal}

\newcommand\Tr{\text{\upshape Tr}}%lzhi add

\newcommand{\ZZ}{{\mathbb Z}}
\newcommand{\NN}{{\mathbb N}}
\newcommand{\QQ}{{\mathbb Q}}
\newcommand{\RR}{{\mathbb R}}

\newcommand{\sdp}{semidefinite program\xspace}
\newcommand{\sdps}{semidefinite programs\xspace}
\newcommand{\sdping}{semidefinite programming\xspace}

\newcommand{\psd}{positive semidefinite\xspace}
\newcommand{\psdness}{positive semidefiniteness\xspace}
\let\nonnegative=\psd

\newcommand{\soses}{sums-of-squares\xspace}
\newcommand{\sos}{sum-of-squares\xspace}

\newcommand{\ess}[2]{{}_{#1,#2}}

\newcommand{\RSOS}[1]{\text{\upshape SOS}/\text{\upshape SOS}_{#1}} %ek3 add upshape
\newcommand{\RSOSdeg}[1]{\text{\upshape SOS}/\text{\upshape SOS}_{\text{\upshape%ek3 add upshape
deg$\le${}$#1$}}}
\newcommand{\SOS}[1]{\text{\upshape SOS}_{#1}} %ek3 add upshape
\newcommand{\calT}[2]{\text{\upshape Terms}[{#1};\text{\upshape deg$\le$}{#2}\hspace{1pt}]} %ek add upshape

\DeclareMathSymbol{\bincircledS}     {\mathrel}{AMSa}{"73}
\newcommand{\symm}[2]{\mathbb{S}#1^{#2\times #2}}
\newenvironment{proof}{\textit{Proof. }}{\hfill $\square$}

\chardef\ttlb="7B % { in \tt
\chardef\ttrb="7D % } in \tt
\chardef\ttti="7E % ~ in \tt

\title{%
Certificates of Impossibility of Hilbert-Artin Representations
of a Given Degree for Definite Polynomials and Functions\raisebox{0ex}{*}%
} %Feng add
\author{%
Feng Guo${}^{1,2}$,
Erich L. Kaltofen${}^1$,
and Lihong Zhi${}^2$
\\
\\% empty line
\small \llap{${}^1$}%
Dept.\ of Mathematics, North Carolina State University,
\\[-0.5ex]
\small Raleigh, North Carolina 27695-8205, USA
\\[-0.5ex]
\small {\ttfamily
kaltofen@math.ncsu.edu};
\url{http://www.kaltofen.us}
\\
\small \llap{${}^2$}%
Key Laboratory of Mathematics Mechanization, AMSS
\\[-0.5ex]
\small Beijing 100190, China
\\[-0.5ex]
\small {\ttfamily
\ttlb fguo,lzhi\ttrb @mmrc.iss.ac.cn};
\url{http://www.mmrc.iss.ac.cn/~lzhi/}
}%author

\begin{document}

\maketitle

\def\thefootnote{\fnsymbol{footnote}}
\footnotetext[1]{%<<<
\footnotesize
%\scriptsize
This material is based on work supported in part
by the National Science Foundation under Grants
CCF-0830347 and CCF-1115772 (Kaltofen),
and by NCSU as host for Guo's study abroad program.
\par
Feng Guo and Lihong Zhi are supported by  NKBRPC 2011CB302400 and
the Chinese National Natural Science Foundation under Grants
91118001, 60821002/F02, 60911130369 and 10871194.
%lzhi was
%supported by the Chinese National Natural Science Foundation under
%Grants 60821002/F02, 60911130369 and 10871194.
}%footnotetext %>>>

\begin{abstract}\noindent %<<<
We deploy numerical \sdping and conversion to exact rational inequalities
to certify that for a \psd input polynomial or rational function,
any representation as a fraction of \soses of polynomials
with real coefficients must contain polynomials in the denominator
of degree no less than a given input lower bound.
By Artin's solution to Hilbert's 17th problems, such representations always
exist for some denominator degree.  Our certificates of infeasibility %lzhi was: infeasibilty
are based on the generalization
of Farkas's Lemma to \sdping.

The literature has many famous examples of impossibility  %lzhi was:
of SOS representability %lzhi was: representablity, 
including Motzkin's, Robinson's, Choi's and Lam's polynomials, and 
Reznick's lower degree
bounds on uniform denominators, e.g., powers of the \sos of each variable.  Our
work on exact certificates for \psdness allows for non-uniform denominators,
which can have lower degree and are often easier to convert to exact
identities.  Here we demonstrate our algorithm by computing certificates of
impossibilities for an arbitrary \sos denominator of degree~$2$ and~$4$ for
some symmetric sextics in~$4$ and ~$5$ variables, respectively.  We can
also certify impossibility of base polynomials in the denominator of restricted
term structure, for instance as in Landau's reduction by one less variable.
\end{abstract} %>>>

%\begin{keyword}
%\end{keyword}

\section{Introduction}
\label{sec:intro}

%\paragraph{Motivation and Problem statement}

%For a given polynomial $f\in \QQ[X]$ and integer $e\ge 0$,
%how to certify that $f$ is not in $\RSOS{S[X;e]}$?

The Farkas Lemma of linear programming can be employed to
construct certificates of infeasibility, in its simplest form
of an inconsistent system of linear equations \citeP{GLS98},
in linear programming of a system of linear inequalities,
and in \sdping of a system of linear equations
with semidefiniteness constraints on the solution.
A polynomial is not a sum-of-squares of polynomials (SOS)
if the corresponding semidefinite program is infeasible.
Thus the Farkas Lemma produces a certificate that a polynomial is not an SOS,
the separating hyperplane \citep{AhmadiParriloNotSOSC}.

Motivated by our SOS certificates for global optima of polynomials
and rational functions
\citeP{%
KLYZ09%
,KYZ09%
,HKZ10%
}
(see also
\citep{%
PW98%
,Harrison07%
,PePa08%
} for earlier work),
we extend those impossibility certificates to Hilbert-Artin
representations of a given denominator degree:  by Emil Artin's
Theorem \citep{Artin27}, every real \psd rational function is
a fraction of two \soses of polynomials. We write for an
$f(X_1,\ldots,X_n)\in K(X_1,\ldots,X_n)$, where $K \supseteq \QQ$
is a subfield of the real numbers,
$$
  f\succeq 0 \text{\quad if\quad}
  \forall \xi_1,\ldots,\xi_n\in\RR\colon f(\xi_1,\ldots,\xi_n) \not< 0.
$$
Note that at a real root of the denominator of $f$ its value is undefined,
hence $\not< 0$.  Artin's original theorem stipulates that
\begin{equation}\label{eq:artin}
\forall f \succeq 0\colon \exists u_1,\ldots,u_l,w\in K[X_1,\ldots,X_n]\colon
f = \frac{1}{w^2} \sum_{i=1}^l u_i^2.
\end{equation}
If $f$ is a polynomial $\in K[X_1,\ldots,X_n]$, %ek []
one may eliminate
one variable from the denominator $w$, that is, construct from
(\ref{eq:artin}) a representation with $w_{\text{new}} \in
K[X_1,\ldots,X_{n-1}]$, which Artin in his 1927 paper attributes
to Edmund Landau.  The reduction can be accomplished with the same
number of $l_{\text{new}} = l$ squares, which in \citep{Raj93} is
attributed to J. W. S. Cassels.  In both constructions the degree
of $w_{\text{new}}$ is substantially larger than that of $w$.  As
is customary, if a \psd polynomial $f$ allows a representation
(\ref{eq:artin}) with $w = 1$, we shall call $f$ a
{\itshape\sos\/} (SOS).  In general, however, 
as already David Hilbert has shown in 1888,
\psd polynomials are not SOS \citep{Chesi07,Blek09}.
%ek The computation of such (polynomial) SOS
%ek representations for \psd polynomials $f$ is also investigated in
%ek \citep{PW98, Harrison07,PePa08}. %lzhi add: Power98.

In order to minimize the numerator
and denominator degrees, we seek
\begin{equation}\label{eq:har}
u_1,\ldots,u_l,\>v_1,\ldots,v_{l^\prime}\in \RR[X_1,\ldots,X_n]
\text{ such that }
f = \frac{\sum_{i=1}^l u_i^2}{\sum_{j=1}^{l^\prime} v_j^2}.
\end{equation}
We shall call (\ref{eq:har}) a {\itshape Hilbert-Artin representation\/} of $f$, which
constitutes an SOS proof for $f \succeq 0$.  By allowing an SOS as the
denominator polynomial, one then can construct such proofs with a possibly
smaller degree than the common denominator $w^2$ in (\ref{eq:artin}){}.  For
instance, for the Motzkin polynomial $\max_j\{\deg(v_j)\} \le 1$ suffices in
(\ref{eq:har}), but $\deg(w) \le 1$ is impossible in (\ref{eq:artin})
\citeP[Section~1]{KLYZ09}.

It is not known if minimal degree denominator SOSes can always have
coefficients in $K$, as is the case in Artin's original theorem
(\ref{eq:artin}){}.  A special case is when $f$ is an SOS of polynomials
($w=1$), and the existence of $u_i$ with all coefficients in $K$ for all $i$ is
conjectured (Sturmfels;
cf.\ \citep{%
Hillar09%
,Ka09:email%
,Quarez09%
,Scheid09%
}).
Our method can certify ``absolute'' impossibility by SOSes,
that is, for coefficients from all possible subfields of $\RR${}.
Our certificates are rational, that is, they have their scalars in $\QQ$. %lzhi was: $K${}.
The problem whether there exists a representation of a given
degree with coefficients in $\QQ$ %lzhi was: $K$
appears to be
decidable \citep{SZ10}.

As in \citeP{KLYZ09}, we compute our certificates, the separating
hyperplanes in Farkas's Lemma, by first computing a numerical
approximation %fguo was: by a high-accuracy 
numerical \sdp solver
%fguo was: \citep{SDPToolsGUO}, 
and then converting the numerical scalars to
exact rational numbers.  
%fguo add:
For ill-posed polynomials (see Example \ref{ex::illposed} below), 
high-accuracy \sdp solver \citep{SDPToolsGUO} is needed. %
The separating hyperplane is the strictly
feasible solution to a \sdp whose objective function tends to
$-\infty$.  We compute such a strictly feasible solution by the
Big-M method \citep{vandenberghe96semidefinite}.  The \sdps in
\citep{AhmadiParriloNotSOSC} and ours certify infeasibility of
SOSes, which has been  %lzhi was: be
generalized to infeasibility of arbitrary
linear matrix inequalities \citep{KLSchw11}.
%lzhi add: Chesi's paper
%ek A matrix
%ek characterization of the set of homogeneous polynomials that are
%ek positive but not SOS was given in \citep{Chesi07} based on
%ek eigenvectors and eigenvalues decomposition.

We have tested our method on polynomials from the literature.  In particular,
we show that the SOS proofs of \psdness in \citeP{KYZ09} indeed require
denominators for three polynomials.  The {\itshape ArtinProver\/} program
\citeP{KLYZ09} successfully introduced denominators not only for purpose of
handling inequalities that do not allow a polynomial SOS proof, but also for avoiding
possible non-rational SOSes, to which the \sdp solvers may have converged in
the case where the Gram matrix is intrinsically rank deficient (our ``hard
case'' \citeP{KLYZ09}){}.  Our impossibility certificates show that for the proof
of the Monotone Column Permanent conjecture in dimension~$4$, actually the
former is the case.

A final problem is to explicitly construct a \psd polynomial for which the
Hilbert-Artin representation (\ref{eq:har}) must have $\deg(\sum_j v_j^2) \ge 4$.
Bruce Reznick in 2009 has kindly provided us with the challenges raised in
\citeP[Section~7]{ESS}: how necessarily high must be the powers $(x_1^2 +
\cdots + x_n^2)^r$ in the (uniform) denominators such that a family $f_{n,k}$
(see Example~\ref{ex:ess} below)
of even symmetric sextics in $n$ variables is an SOS, where $2 \le k \le n-2$?
In \citeP{ESS} it is proven that for $f\ess{4}{2}$ one has $r=2$.  We can compute
certificates that show that for $f\ess{4}{2}$, $f\ess{5}{2}$, %fguo change: and
%$f\ess{5}{3}$ the degree
%lower bound $\ge 4$ even holds 
$f\ess{6}{2}$, the degree lower bound $\ge 4$ and for 
$f\ess{5}{3}$, $f\ess{6}{4}$, the lower bound $\ge 6$ 
even hold %
for {\itshape any\/} denominator $\sum_j v_j^2$
in (\ref{eq:har}){}.  %fguo delete: {\itshape SDPTools\/} is not able to certify for
%$f\ess{5}{2}$ a bound $\deg(\sum_j v_j^2)\ge 6$, but we expect this to become
%possible with a new high-accuracy \sdp solver, which is announced for Maple~16.

\paragraph{Notation} Throughout this paper, $\NN$ denotes the set of
nonnegative integers and we set $\NN^n_t=\{\alpha\in \NN^n\ |\
|\alpha|=\sum^{n}_{i=1}\alpha_i \le t\}$ for $t\in \NN$.
$\RR[X]=\RR[X_1,\ldots,X_n]$ denotes the ring of polynomials in variables
$X= %ek add
(X_1,\ldots,X_n)$ with real coefficients. %lzhi ??? was:rational coefficients.
Given a polynomial
$f={\sum_{\alpha}}f_{\alpha}X_1^{\alpha_1}\ldots X_n^{\alpha_n}
\in \RR[X]$, let $\text{supp}(f)=\{X_1^{\alpha_1}\ldots
X_n^{\alpha_n}\ |\ c_{\alpha}\neq 0\}$, i.e., %ek add comma
the set of the
support terms of $f$. Denote by $\deg(f)$ the total degree
of $f$. %lzhi delete: Let SOS be the set of a sum of  squares of polynomials in
%$\RR[X]$.
Given $n \ge 1$ %ek was >0
and $e\ge 0$, let
$\calT{X}{e}=\{X_1^{e_1}\ldots X_n^{e_n}\ |\sum^{n}_{i=1}e_i\le
e\}$, %lzhi was: \ e_1+\cdots+e_n\le e
i.e., %ek add comma
the set of all terms %ek was monomials
of total %ek add
degree
$\le e$ in the $n$ variables $X$. %ek add
For a given %ek add
subset $\mathcal{T} %ek was: S[X;e]
\subseteq \calT{X}{e}$, we introduce the following notation %ek denote
for a term-restricted SOS,
$\SOS{\mathcal{T} %ek2 S[X;e]
}=\{\sum v_j^2\ |\ v_j \in \RR[X], \text{supp}(v_j)\subseteq
\mathcal{T} %ek2 S[X;e]
\}$, %ek remove \subseteq\calT{X}{e},
and the following notation for a denominator term-restricted Hilbert-Artin representation,
\begin{equation}
\RSOS{\mathcal{T} %ek2 S[X;e]
}=
\Big\{
\sum u_i^2\Big/\sum v_j^2 \Bigl.\;\Bigr|\;
%ek2 remove \forall i,j\colon
u_i,v_j \in \RR[X],
\forall j\colon %ek2 add
\text{supp}(v_j)\subseteq \mathcal{T} %ek2 S[X;e]
\Big\}. %ek remove \subseteq\calT{X}{e}. %Let $d:=\lceil
%e+\deg{(f)}/2\rceil$ then $\deg{u_i}\le d$.
\end{equation}
Finally, we write the shorthand $\RSOSdeg{2e} = \RSOS{\calT{X}{e}}$.

%fguo was:
%By $\symm{\RR}{n}\subseteq \RR^{n\times n}$ %ek change $\mathcal{S}^n$
%we denote the subspace of real symmetric $n\times n$ matrices.
%For a matrix $W\in \symm{\RR}{n}$, $W \succeq 0$ means $W$ is positive
%semidefinite.  The bold number zero ${\mathbf %ek was \bf
%0}$ denotes the zero matrix and $I$
%denotes the identity matrix. %For a vector $v$, $v'$ denotes
%%the transformation of $v$.

By $\symm{\RR}{k}$ %ek3: PLEASE! $\mathcal{S}^k$
we denote the subspace of real symmetric $k\times k$ matrices.
For a matrix $W\in \symm{\RR}{k}$, $W \succeq 0$ means $W$ is positive semidefinite.
%ek4 tried but... The cone of all $k \times k$ matrices $W \succeq 0$ is denoted by $\symmpsd{\RR}{k}$.
The bold number zero ${\mathbf %ek was \bf 0}$
0}$ denotes the zero matrix, and $I$ denotes the identity matrix.
%For a vector $v$, $v'$ denotes
%the transformation %ek4 transposed??? of $v$.

\ifskiptext\else
\section{Hilbert-Artin representation of positive semidefinite polynomials}\label{sec::polys}

\subsection{Rational Function Sum-of-Squares and Semidefinite Programming}
%fguo add:
For a given subset $\mathcal{T}\subseteq \calT{X}{e}$,
note that $f\in\RSOS{\mathcal{T}}$ if and only if
\[
    0=\underset{i}{\overset{l}{\sum}} u_i(X)^2+(-f)\underset{j}{\overset{l'}{\sum}}
    v_j(X)^2, %lzhi was:  v_i(X)^2
\]
for some polynomials $u_i(X), v_j(X)\in\RR[X]$ %lzhi was: v_j(x)
with
$\text{supp}(v_j)\in \mathcal{T}$. Consider the following set:
%%\begin{equation}\label{optproblem}
%%      \begin{aligned}
%%            r^*:=\inf & \ \text{Trace}(W^{[1]})+\text{Trace}(W^{[2]})\\
%%              s.t. &\ {m_{\calT{X}{d}}}^TW^{[1]}m_{\calT{X}{d}}=f(X)\cdot {m_{S[X;e]}}^T W^{[2]} m_{S[X;e]}\\
%%      &\  W^{[1]} \succeq 0,\ W^{[2]}\succeq 0, \\
%%      &\  \text{Tr}(W^{[2]})=1,
%%        \end{aligned}
%%\end{equation}
\begin{equation}\label{optproblem}
        \left\{\begin{array}{c|c}[W^{[1]},W^{[2]}]&
        \begin{aligned}
         &\ {m_{\calT{X}{d}}}^TW^{[1]}m_{\calT{X}{d}}=f(X)
     \cdot {m_{\mathcal{T}}}^T W^{[2]} m_{\mathcal{T}}\\
     &\ W^{[1]} \succeq 0,\ W^{[2]}\succeq 0,\ \text{Tr}(W^{[2]})=1
        \end{aligned}
    \end{array}
    \right\},
\end{equation}
%\begin{equation}\label{optproblem}
%        \left\{ [W^{[1]},W^{[2]}]\ \Big\vert
%       \begin{aligned}
%         &\ {m_{\calT{X}{d}}}^TW^{[1]}m_{\calT{X}{d}}=f(X)
%    \cdot {m_{S[X;e]}}^T W^{[2]} m_{S[X;e]}\\
%%   &\ W^{[1]} \succeq 0,\ W^{[2]}\succeq 0,\ \text{Tr}(W^{[2]})=1
%        \end{aligned}
%   \right\},
%\end{equation}
where $m_{\mathcal{T}}$ and $m_{\calT{X}{d}}$ denote the column vectors
which consist of the elements in $\mathcal{T}$ and $\calT{X}{d}$,
respectively.
%lzhi add
Here and hereafter, we let $d=\lceil e+\deg(f)/2\rceil$, and
therefore,
\[
\left\{X^{\alpha+\beta}\ |\
X^{\alpha},X^{\beta}\in\calT{X}{d}\right\} \supseteq
\left\{X^{\alpha+\beta+\gamma}\ |\ X^{\gamma}\in \text{supp}(f),
X^{\alpha},X^{\beta}\in \mathcal{T}\right\}.
\]
The last constraint $\text{Tr}(W^{[2]})=1$ is added to enforce
that $W^{[2]}\neq{\mathbf 0}$.
%Here $\text{Trace}(W^{[1]})+\text{Trace}(W^{[2]})$ acts as a dummy objective
%function that is commonly used in SDP for optimization problem without an objective function.
%lzhi delete: Let $d=\lceil e+\deg{(f)}/2\rceil$,
%lzhi add
\begin{prop}
%lzhi was: Hence
We have $f\notin \RSOS{\mathcal{T}}$ if and only if  the set %lzhi add: the
{\upshape(\ref{optproblem})} is empty.
\end{prop}
%Now we rewrite the problem (\ref{optproblem}) to be the following

Now we review the following standard {\itshape Semidefinite Program} (SDP)
(see \citep{vandenberghe96semidefinite}),
\begin{equation}\label{StandardSDP}
\begin{array}{llll}
\underset{W\in \symm{\RR}{k}}{\sup}&\ -C\bullet W & ~~ \underset{y\in \mathbb{R}^l}{\inf}&\ b^Ty\\
 s.t.&\ A_i\bullet W=b_i,\ i=1\cdots l, & ~~ s.t. &  C+\sum_{i=1}^{l} y_iA_i\succeq
 0. \\
     &\ W\succeq 0. &
\end{array}
\end{equation}

%lzhi was:
%\begin{equation}\label{StandardSDP}
%\begin{array}{lr}
%\begin{aligned}
%\underset{W\in \symm{\RR}n}{\sup}&\ -C\bullet W\\
% s.t.&\ A_i\bullet W=b_i,\ i=1\cdots l,\\
%     &\ W\succeq 0.
%\end{aligned}
%&
%\begin{aligned}
%\underset{y\in \mathbb{R}^l}{\inf}&\ b^Ty\\
% s.t. &\ C+\sum_{i=1}^{l} y_iA_i\succeq 0.\\ %lzhi was: \sum y_iA_i
% \\ %lzhi add:\\
%\end{aligned}
%\end{array}
%\end{equation}
For symmetric matrices $C,\ W$, the scalar product in %lzhi was: on
 $\RR^{n\times n}$ space is defined as
\[
    C\bullet W=\langle C, W\rangle={\sum_{i=1}^n}{\sum_{j=1}^n}c_{i,j}w_{i,j}=\Tr CW.
\]
%lzhi was:  C\bullet W=\langle C, W\rangle=\underset{i}{\sum}\underset{j}{\sum}c_{i,j}w_{i,j}=\text{Trace}(CW).
Let
\begin{equation}\label{def::G}
{m_{\calT{X}{d}}}^T W^{[1]} m_{\calT{X}{d}}=\underset{\alpha}{\sum}(G^{[\alpha]}\bullet W^{[1]})X^{\alpha},
\end{equation}
where $G^{[\alpha]}$ are scalar symmetric matrices and $X^{\alpha}$ are all
possible terms appearing in the polynomial of degree $\le 2d$. Similarly,
%for $\beta:=(\beta_1,\ldots,\beta_n)\in \NN^n$,
let
\[
(-f(X))\cdot {m_{\mathcal{T}}}^T W^{[2]} m_{\mathcal{T}}=\underset{\beta}{\sum}(H^{[\beta]}\bullet W^{[2]})X^{\beta},
\]
where $H^{[\beta]}$ are symmetric matrices and $X^{\beta}$ are all possible
terms appearing in the product of $(-f(X))$ with a polynomial of degree $\le 2e$.
%Now we can rewrite the problem (\ref{optproblem}) to be the following block SDP:
Now we consider the following block SDP:
\begin{equation}\label{primal}
        \begin{aligned}
               %r^*:=\sup & \ -C\bullet W\\
               \underset{W\in\symm{\RR}{k}}{\sup} & \ -C\bullet W\\
        s.t. &\ \left[\begin{array}{c} \vdots \\ A^{[\alpha]}\bullet W \\ \vdots \\ A\bullet W
        \end{array}\right]=\left[\begin{array}{c}\vdots\\0\\ \vdots \\ 1\end{array}\right],
        &  W \succeq 0.\\
        \end{aligned}
\end{equation}
where %lzhi was: $C:={\bf 0}$,
%fguo add:
$k={\tbinom{n+d}{d}+\tbinom{n+e}{e}}$,
\[
    C:=\left[\begin{array}{cc} {\bf 0} & \\ & {\bf 0}  \end{array}\right], \ \  W:=\left[\begin{array}{cc} W^{[1]} & *  \\  * & W^{[2]} \end{array}\right],\  \
    A^{[\alpha]}:=\left[\begin{array}{cc} G^{[\alpha]} & \\ & H^{[\alpha]}\end{array}\right],\  \
    A:=\left[\begin{array}{cc} {\bf 0} & \\ & I  \end{array}\right] \ \
\]
and $\alpha$ ranges over %lzhi delete: set
$\NN^n_{2d}$.   The matrix  $C$ can be chosen as a random
symmetric matrix. We set it to be a zero matrix  only for the
convenience of discussions below.
%lzhi: was  Since $C$ could be chosen as any
%symmetric matrix, we let it be zero matrix for the convenience of the proof and
%certification we will give in the following sections.
 For all  block positive semidefinite %lzhi add: positive semidefinite, delete PSD
matrices appearing in the present paper, we % lzhi was: fill the positions of some
% blocks
% with
use the symbol  $*$ to indicate that the associated elements could
be any real numbers such that the whole matrices are still
positive semidefinite %lzhi was: PSD
 and
leave some positions blank to indicate that the associated blocks
are zero matrices.

%lzhi add
\begin{prop}
We have $f\notin \RSOS{\mathcal{T}}$ if and only if SDP {\upshape(\ref{primal})}
is infeasible. %lzhi was: infeasibility.
\end{prop}

\subsection{Dual Problem and Certification}
Before we consider the dual problem of (\ref{primal}), let us
review some definitions about  moment matrices and  localizing
moment
matrices. %lzhi add: s
Given a sequence $y=(y_{\alpha})_{\alpha \in \NN^n}\in
\RR^{\NN^n}$, its {\itshape moment matrix} is the (infinite) real symmetric %lzhi add:
matrix $M(y)$ indexed by $\NN^n$, with $(\alpha, \beta)$th entry
$y_{\alpha+\beta}$, for $\alpha,\beta\in \NN^n$. Given an integer %lzhi was: a
$t\ge 1$ and a truncated sequence $y=(y_{\alpha})_{\alpha \in %lzhi add: a
\NN_{2t}^n}\in \RR^{\NN_{2t}^n}$, its {\itshape moment matrix of
order t} is the matrix $M_t(y)$ %lzhi delete:indexed by $\NN_t^n$,
with
$(\alpha, \beta)$th entry $y_{\alpha+\beta}$, for $\alpha,\beta\in
\NN_t^n$. For a given polynomial $q\in \RR[X]$, %lzhi add: polynomial
if the $(i,j)$th entry of $M_t(y)$ is $y_{\beta}$, then the $t$th
{\itshape localizing moment matrix} of $q$ is defined by
\[
    M_t(qy)(i,j):=\underset{\alpha}{\sum}q_{\alpha}y_{\alpha+\beta}.
\]
More details about moment matrices, see %lzhi was: the moment sequence and moment matrix, see
\citep{Lasserre01globaloptimization, Lasserre09,Laurent_sumsof}. %lzhi add: Lasserre09,

According to (\ref{StandardSDP}), the dual problem of (\ref{primal}) %lzhi add": the
is
\begin{equation}\label{dual}
        \begin{aligned}
             s^*:=\underset{(y,s)\in\RR^{m+1}}{\inf} & \ s\\
                %s.t.  &\ M(y,s)+I \succeq 0,\\
                s.t.  &\ M(y,s)\succeq 0,\\
        \end{aligned}
\end{equation}
where
\[
    M(y,s):=\left[\begin{array}{cc}M_d(y) & \\ & M_e((-f)y)+sI\end{array}\right],
\]
$y:=(y_{\alpha})_{\alpha \in \NN_{2d}^n}\in \RR^{\NN_{2d}^n}$ and $m={\tbinom{n+2d}{2d}}$, %lzhi was: $m=|\NN_{2d}^n|$.
$M_d(y)$ is a truncated moment matrix of order $d$ and
$M_e((-f)y)$ is the $e$th localizing  moment matrix.
%lzhi was:
%$M_d(y)$ and $M_e((-f)y)$ are moment matrix and localizing moment
%matrix as
%defined above.%$y_0$ is the first element in the sequence

The next lemma shows that the  problem (\ref{dual}) is strictly
feasible. The proof is similar to the one given in
\citep[Proposition3.1]{Lasserre01globaloptimization}.
\begin{lemma}\label{lem::SF}
There exists $(\tilde{y},\tilde{s})\in\RR^{m+1}$ such that %lzhi delete: \tilde{p}
$M_d(\tilde{y})\succ 0$ and $M_e(-f\tilde{y})+\tilde{s}I\succ 0$.
\end{lemma}
\begin{proof}
%lzhi delete:
%As is proved in  of proposition 3.1 in
%\citep{Lasserre01globaloptimization},
Let $\mu$ be a probability measure on $\RR^n$ with a {\itshape
strictly positive} density $h$ with respect to Lebesgue measure
such that
\[
    \tilde{y}_{\alpha}:=\int X^{\alpha} d\mu < \infty.
\]
For any polynomial $q(X)\in\RR[X]$ with $\text{supp}(q)\in \calT{X}{d}$, let
$\text{vec}(q)$ denote its sequence of coefficients in the monomial basis
$\calT{X}{d}$. We have
\begin{equation*}
\begin{aligned}
\langle \text{vec}(q)^T,M_d(\tilde{y})\text{vec}(q)\rangle&=\int q(x)^2\mu(dx)\\
&=\int q(x)^2h(x)dx\\
&>0\ \text{whenever}\ q\neq 0, %lzhi was:.
\end{aligned}
\end{equation*}
which implies $M_d(\tilde{y})\succ 0$. %Then we choose $\tilde{y}=\frac{1}{2}y$ and
Take
\[
    \tilde{s}>-\lambda_{\text{min}}(M_e((-f)\tilde{y})),
\]
then $M_e(-f\tilde{y})+\tilde{s}I\succ 0$.
\end{proof}

%ek3 rewrite Farkas
%ek3 For standard SDPs in (\ref{StandardSDP}), we have the following {\itshape
%ek3 Semidefinite Farkas' Lemma}:
%ek3 \begin{lemma}\label{lem::Farkas}{\upshape \citep[Lemma 2.3]{Alizadeh93interiorpoint}} %lzhi add:
%ek3  Given a %lzhi add: a
%ek3  set $\{A_i\in \symm{\RR}{k},\ i=1,\ldots, l\}$ %lzhi was: \cdots
%ek3  and a vector $b\in
%ek3 \RR^l$, %lzhi add: a
%ek3  let $y\in\RR^l$ be a vector  %lzhi was: there be a vector $y\in\RR^l$
%ek3   such that $\sum_{i=1}^l
%ek3 y_iA_i\succ 0$, then the following are equivalent:
%ek3 \begin{enumerate}[$(a)$]
%ek3 \item There exists a %lzhi add: a
%ek3  symmetric matrix $W\succeq 0$ such that
%ek3 $A_i\bullet W=b_i,\ i=1,\ldots,l$; \item $b^Ty\ge 0$ holds for %lzhi was: \cdots
%ek3 every vector $y\in \RR^l$ such that
%ek3 $%lzhi was: \underset{i=1}{\overset{l}
%ek3 {\sum_{i=1}^l}y_iA_i\succeq 0$.
%ek3 \end{enumerate}
%ek3 \end{lemma}
For standard SDPs in (\ref{StandardSDP}), we have the following important duality fact.
\begin{lemma}\label{lem::Farkas}
{\upshape \citep[Lemma 2.3; {\scshape Semidefinite Farkas Lemma\/}]%
{Alizadeh93interiorpoint}}\newline
Let $A_i\in \symm{\RR}{k}$ for all $i=1,\ldots, l$ and let $b\in \RR^l$.
Suppose there exists a vector $y\in\RR^l$
such that $\sum_{i=1}^l y_iA_i\succ 0$.
Then exactly one of the following is true:
\begin{enumerate}[{\upshape 1.}]
\item
There exists a \psd symmetric matrix $W\in\symm{\RR}{k}$, $W \succeq 0$,
such that $A_i\bullet W=b_i$ for all $i=1,\ldots,l$;
\item
There exists a vector $\hat y\in \RR^l$ such that $\sum_{i=1}^l \hat
y_i A_i \succeq 0$ and $b^T \hat y < 0$. %lzhi3 was: y
\end{enumerate}
\end{lemma}
We call the vector $\hat y$ {\itshape Farkas's %ek4 Farkas' is correct, but I prefer the also
%ek4                                                correct Farkas's;  I always write Gauss's Lemma, etc.
certificate vector of infeasibility.}
For other forms of the Farkas %ek4 Farkas is used as an adjective here, hence "the"
Lemma, see \citep[Section 4.2]{COEDG}.

%lzhi replace: Farkas" Lemma by Lemma \ref{lem::Farkas}
Note that Lemma~\ref{lem::SF} implies that the assumption in %Farkas'
the Farkas %ek4 name AND \ref, why not?
Lemma~\ref{lem::Farkas} is satisfied in SDPs (\ref{primal}) and
(\ref{dual}). Then we have our main result:
\begin{theorem}\label{th::main1}
Given a polynomial $f\in \QQ[X]$ and an integer $e\ge 0$, %lzhi add: a, an
 let
$d=\lceil e+\deg(f)/2\rceil$, then for any subset
$\mathcal{T}\subseteq \calT{X}{e}$, the following are equivalent:
\begin{enumerate}[\upshape 1.]
\item\label{1} $f\notin \RSOS{\mathcal{T}}$,
\item\label{2} There exists a rational vector $\hat{y}=(\hat{y}_\alpha)\in \QQ^{m}$
with $m=\tbinom{n+2d}{2d}$ such that %lzhi rewrite
$M_d(\hat{y})\succeq 0$ and $\ M_e(f\hat{y})\prec 0$.
\end{enumerate}
%lzhi was: where $m$ is the number of elements in set $\NN^n_{2d}$.
\end{theorem}
\begin{proof}
By employing the Farkas %ek4 add
Lemma~\ref{lem::Farkas} to SDPs (\ref{primal}) and
(\ref{dual}), we have that $f\notin \RSOS{\mathcal{T}}$ if and only if
there exists $p'=(y',s')\in \RR^{m+1}$ for (\ref{dual}) such that
$M(y',s')\succeq 0$ and $s'<0$. Now we prove that $p'$ can be
chosen to be rational.

Let $\tilde{p}=(\tilde{y},\tilde{s})$ be
the strictly feasible point constructed %lzhi add:
 in Lemma \ref{lem::SF}.
For $0<t\le 1$, let $\bar{y}=(1-t)y'+t\tilde{y}$ and
$\bar{s}=(1-t)s'+t\tilde{s}$, then $M(\bar{y},\bar{s})\succ 0$.
%lzhi was: Fix $t$ such that $\bar{s}<0$.
Since $s'<0$, it is always possible to choose  a rational number
$t$ such that $\bar{s}<0$.
 Then there exists $\varepsilon>0$ such that
for all $p=(y,s)\in B_{\bar{p}}(\varepsilon)$ where
$B_{\bar{p}}(\varepsilon)$ is a ball with  center $\bar{p}$ and
radius $\varepsilon$, we have $M(y,s)\succeq 0$.  Taking
$\varepsilon<\frac{1}{2}|\bar{s}|$, there always exists a point
$\hat{p}=(\hat{y},\hat{s})\in B_{\bar{p}}(\varepsilon)$ such that
$\hat{p}\in \QQ^{m+1}$, %lzhi add:,
 $M_d(\hat{y})\succeq 0,\
M_e(-f\hat{y})+\hat{s}I\succeq 0$ and $\hat{s}<0$ which implies
$M_e(f\hat{y})\prec 0$.
\end{proof}

\subsection{Moment matrices and linear forms on $\RR[X]$}

In this section, we give an interpretation of %lzhi delete: above
 our infeasibility %lzhi add
  certification  using  linear forms on $\RR[X]$.
 %lzhi delete
  %from the
%relationship of moment matrices and linear forms on $\RR[X]$ which can be
%found in \citep{Laurent_sumsof}.

Given $y\in \RR^{\NN^n}$, we define the linear form $L_y \in {(\RR[X])}^*$ by
\begin{equation}\label{linearform}
    L_y(f):=y^T\text{vec}(f)=\underset{\alpha}{\sum}y_{\alpha}f_{\alpha}\
    \text{for}~ f=\underset{\alpha}{\sum}f_{\alpha}X^{\alpha}\in \RR[X],
\end{equation}
where $\text{vec}(f)$ denotes its sequence of coefficients. %lzhi delete: in the
%monomial basis of $\RR[X]$.
\begin{lemma}\label{lem::linearform}{\upshape\citep[Lemma
4.1]{Laurent_sumsof}} %lzhi add:
Let $y\in \RR^{\NN^n}$, $L_y \in {(\RR[X])}^*$ the associated
linear form,  %lzhi delete: (\ref{linearform}),
and let $f,g,h\in \RR[X]$.
\begin{enumerate}[\upshape 1.]
\item $L_y(fg)=\text{vec}(f)^TM(y)\text{vec}(g)$; in particular,
$L_y(f^2)=\text{vec}(f)^TM(y)\text{vec}(f)$, $L_y(f)=\text{vec}(1)^TM(y)\text{vec}(f)$.
\item $L_y(fgh)=\text{vec}(f)^TM(y)\text{vec}(gh)=\text{vec}(fg)^TM(y)\text{vec}(h)=
\text{vec}(f)^TM(hy)\text{vec}(g)$.
\end{enumerate}
\end{lemma}

Now we have the following statement which is equivalent to Theorem
\ref{th::main1}:

\begin{theorem}\label{th::main2}
Given a polynomial $f\in \QQ[X]$ and an integer $e\ge 0$, %lzhi add: a, an
let
%$\RR[X]_{2d}:=\{p\in\RR[X]\ |\ \text{supp}(p)\in \calT{X}{2d}\}$
%where
 $d=\lceil e+\deg(f)/2\rceil$,
  then for any subset $\mathcal{T}\subseteq \calT{X}{e}$,
the following are equivalent:
\begin{enumerate}[\upshape 1.]
\item $f\notin\RSOS{\mathcal{T}}$,
\item There exists a rational vector $\hat{y}\in \QQ^{m}$ with
$m={\tbinom{n+2d}{2d}}$, and the associated linear form
$L_{\hat{y}}\in {(\RR[X]_{2d})}^*$ such that for any polynomials $v,
u\in\RR[X]$ with $\text{supp}(v)\in \mathcal{T}$ and $\text{supp}(u)\in
\calT{X}{d}$, we have $L_{\hat{y}}(fv^2)< 0$ and
$L_{\hat{y}}(u^2)\ge 0$.
\end{enumerate}
 %lzhi was: where $m$ is the number of elements in the
%set $\{X^{\alpha+\beta}\ |\ X^{\alpha}, X^{\beta} \in
%\calT{X}{d}\}$.
\end{theorem}
\begin{proof}
By (\ref{dual}) and Theorem \ref{th::main1}, we have that $f\notin
\RSOS{\mathcal{T}}$ if and only if there exists $\hat{y}\in\QQ^m$ such
that $M_d(\hat{y})\succeq 0$ and $M_e(f\hat{y})\prec 0$. According
to %lzhi add:to
Lemma \ref{lem::linearform}, the conclusion follows.
\end{proof}

Now one has a better understanding that the existence of a %lzhi add: a
certificate $\hat{y}$ in Theorem \ref{th::main1} implies $f\notin
\RSOS{\mathcal{T}}$. In fact if $f=\sum u_i^2/\sum v_j^2$ with
$\text{supp}(u_i)\in \calT{X}{d}$ and $\text{supp}(v_j)\in \mathcal{T}$,
then $0\le L_{\hat{y}}(\sum u_i^2)=\sum L_{\hat{y}}(fv_j^2)<0$ which
is a contradiction.

%\subsection{Some special cases}
\begin{remark}
One special case is $e=0$, i.e. we certify that $f$ can not be
written as a rational SOS.   %lzhi was: when $e=0$ in which our problem becomes to certifying a
%given polynomial $f$ being not an SOS.
 According to Theorem %lzhi was: Corollary %lzhi add: to
\ref{th::main2}, $f$ is not an SOS if and only if there is
$\hat{y}\in \QQ^m$ and the associated linear form $L_{\hat{y}}$,
such that $\forall u\in \RR[X]$ with
$\text{supp}(u)\in\calT{X}{\lceil\deg(f)/2\rceil}$,
$L_{\hat{y}}(u^2)\ge 0$ and $L_{\hat{y}}(f)< 0$. This special case
has %lzhi was: is also
also been studied in \citep{AhmadiParriloNotSOSC},  in which
$\hat{y}$ is referred as the {\itshape separating hyperplane}.
\end{remark}

%For other special cases, we can set $m_{S[X;e]}(X)$ in (\ref{optproblem}) to be a subset of
%all the monomials with degree $\le e$ in a subset of the variables $[X_1,\ldots,X_n]$. Then
%we can certify $f\notin \RSOS{S[X;e]}$ with a restrictive support monomials of the denominator.

\section{Computational aspects of the certification}

\subsection{Finding $\hat{y}$ by Big-M method}
Given a polynomial $f\in \QQ[X]$ and an integer $e\ge 0$, %lzhi add: a, an
 note that
$f\notin\RSOS{\mathcal{T}}$ if and only if (\ref{primal}) is infeasible.
From the proof of Theorem \ref{th::main1}, we have
\begin{lemma}
$f\notin\RSOS{\mathcal{T}}$ if and only if $s^*=-\infty$ in {\upshape(\ref{dual})}.
\end{lemma}
%\begin{proof}
%From the proof of Lemma \ref{lem::SF}, we know (\ref{dual}) is strictly feasible.
%Then the conclusion follows the duality theory of SDP that is if one of the primal and
%dual problems is strictly feasible, then duality gap is zero where we allow
%values $\pm \infty$. See \citep{BaV:02,vandenberghe96semidefinite}.
%\end{proof}
To find a %lzhi add: a
certificate $\hat{y}$ in Theorem \ref{th::main1}, we need find a
feasible point of the dual problem (\ref{dual}) %lzhi add: the
 at which the value %lzhi was: values
of its objective function $s$ is negative. %lzhi add: $s$
 We employ the Big-M method
\citep{vandenberghe96semidefinite} to (\ref{primal}) and
(\ref{dual}), and solve the following two modified SDPs %lzhi was: modify them to
\begin{equation}\label{Mprimal}
        \begin{aligned}
               r^*:=\underset{W\in\symm{\RR}{k},w\in\RR}{\sup} & \ -C\bullet (W-w)-\mathcal{M}w\\
        s.t. &\ \left[\begin{array}{c} \vdots \\ A^{[\alpha]}\bullet (W-w)
        \\ \vdots \\ A\bullet (W-w) \end{array}\right]=\left[\begin{array}{c}
        \vdots\\0\\ \vdots \\ 1\end{array}\right],
        &  W \succeq 0, w\ge 0,\\
        \end{aligned}
\end{equation}
\begin{equation}\label{Mdual}
        \begin{aligned}
             s^*:=\underset{(y,s)\in\RR^{m+1}}{\inf} & \ s\\
                %s.t.  &\ M(y,s)+I \succeq 0,\\
                %      &\ \text{Trace}(M(y,s)+I)\le \mathsf{M},\\
                s.t.  &\ M(y,s) \succeq 0,\\
                      &\ \Tr M(y,s)\le \mathcal{M},\\ %lzhi change: Tr-->Trace
        \end{aligned}
\end{equation}
where matrices $C,\ A^{[\alpha]},\ A,\ M(y,s)$ in (\ref{Mprimal})
and (\ref{Mdual}) are  defined as in (\ref{primal}) and
(\ref{dual}). Note that any feasible point of (\ref{Mdual}) is
also feasible to (\ref{dual}). As shown in
\citep{vandenberghe96semidefinite}, (\ref{Mprimal}) and
(\ref{Mdual}) are always strictly feasible and $r^*=s^*\rightarrow
-\infty$ as $\mathcal{M}\rightarrow \infty$. Hence a certificate
$\hat{y}$ is obtained by solving (\ref{Mprimal}) and (\ref{Mdual})
using  interior-point methods.
%by means of interior method.

\begin{algorithm}\label{alg}\quad\\
{\bf Input:} $f\in \QQ[X]$, $e\in\ZZ_{\ge 0}$ and a subset $\mathcal{T}\subseteq \calT{X}{e}$.
\\
{\bf Output:} If $f\notin\RSOS{\mathcal{T}}$, return a certificate
$\hat{y}\in\QQ^m$.
\begin{enumerate}[\upshape I.]
\item Reduce the problem to SDPs {\upshape(\ref{primal})} and {\upshape(\ref{dual})}.
\item Fix a big ${\mathcal M\in\ZZ}$ and modify {\upshape(\ref{primal}), (\ref{dual})} 
to {\upshape(\ref{Mprimal}), (\ref{Mdual})}.
%\item $(P^*)$, $(D^*)\rightarrow$ {\bf SeDuMi} in Matlab or {\bf SDPTools} in Maple.
\item\label{getiter} Solve {\upshape(\ref{Mprimal})} and {\upshape(\ref{Mdual})} 
by  interior-point methods %lzhi was:iteration
until a solution $p_k=(y^k,s^k)$ with $s^k<0$ is obtained.
%\begin{itemize}
%\item If $\hat{s}^k<0$, go to step \ref{chrat}.
%\item Otherwise, let $k:=k+1$ and go \ref{getiter}.
%\end{itemize}
\item\label{chrat} Find a strictly feasible point
$\tilde{p}=(\tilde{y},\tilde{s})$ of {\upshape(\ref{dual})}.
\item\label{getint} Fix $0<t\le 1$ and
$\bar{p}=(1-t)p_k+t\tilde{p}=(\bar{y},\bar{s})$ such that
$\bar{s}<0$. \item\label{getrat} Choose a rational point
$\hat{p}=(\hat{y},\hat{s})\in B_{\varepsilon}(\bar{p})$ where
$\varepsilon<\frac{1}{2}|\bar{s}|$.
%\begin{itemize}
%\item If $\hat{p}$ is feasible and $\hat{s}<0$, return $\hat{y}$.
%\item Otherwise, decrease $\varepsilon$ and go to step \ref{chrat}.
%\end{itemize}
\end{enumerate}
\end{algorithm}

\begin{remark}

In Step \ref{getiter}, provided that the problem is of large size and not
ill-conditioned, we can solve (\ref{Mprimal}) and (\ref{Mdual}) using SDP
solvers in Matlab like SeDuMi \citep{Sturm99} which is very %lzhi delete: fast and 
efficient.  If the problem has small size and an accurate solution %lzhi was:pricise computation 
is needed, Maple package {\itshape SDPTools\/}
\citep{SDPToolsGUO} is a better choice. {\itshape SDPTools}, in which the above
algorithm has been implemented, is a high precision SDP solver based on the
potential reduction method in \citep{vandenberghe96semidefinite}.

\end{remark}
\begin{remark}

In practice, if (\ref{Mprimal}) and (\ref{Mdual}) in Step
\ref{getiter} are precisely computed by  interior-point methods, %lzhi add: the, -point
then the floating-point %lzhi add: -point
 solution $(y^k,s^k)$ is a highly accurate approximation
%lzhi was: approximate
of a strictly feasible point of (\ref{dual}). Hence, without Step
\ref{chrat}, \ref{getint}, \ref{getrat}, one can expect that an
exact certificate can be  obtained by simply rounding $(y^k,s^k)$ to %lzhi delete: be
 a rational feasible solution %lzhi was one which is still feasible
to (\ref{dual}).

\end{remark}

%Our approach is implemented in Maple package SDPTools \citep{SDPToolsGUO}.
%SDPTools is a high precision SDP solver based on the potential reduction method
%in \citep{vandenberghe96semidefinite}.  It employs the  Big-M method to handle
%the requirement of the strictly feasible initial point of the interior-point
%method.% After we get solution
%$(\hat{y},\hat{s})$ such that $M(\hat{y},\hat{s})\succeq 0$ and $\hat{s}<0$,
%we truncate and convert $(\hat{y},\hat{s})$ to rational ($\hat{y},\hat{s}$) and expect
%that $M(\hat{y},\hat{s})\succeq 0$.

\subsection{Exploiting the sparsity}

To reduce computation cost, we can replace $m_{\calT{X}{d}}$ in
(\ref{optproblem}), i.e. the vector of all  terms with degree $\le
d$ by a sparse vector containing part of $m_{\calT{X}{d}}$ due to
the following theorem:

\begin{theorem}\label{th:sparsity}{\upshape\citep[Theorem 1]{Reznick78}} %lzhi add: Theorem 1
  For a polynomial
$p(x)=\sum_{\alpha}p_{\alpha}x^{\alpha}$, we define $C(p)$ as the convex hull
of $\{\alpha |\ p_{\alpha} \neq 0 \}$, then we have  $C(p^2)=2C(p)$; for any
positive semidefinite polynomials $f$ and $g$, $C(f) \subseteq C(f+g)$; if
$f=\sum_jg_j^2$ then $C(g_j) \subseteq \frac{1}{2}C(f)$.  \end{theorem}

\begin{define}
Given a  polynomial $f\in\RR[X]$, an integer $e\ge 0$ and a subset
$\mathcal{T}\subseteq \calT{X}{e}$, let $\cal{C}_{f,\mathcal{T}}$ be the
convex hull of $\{\alpha\in \NN^n\ |\
\alpha=\beta+\gamma_1+\gamma_2, X^{\beta}\in\text{supp}(f),
X^{\gamma_1}, X^{\gamma_2}\in \mathcal{T}\}$.  We define
$\mathcal{G}_{f,\mathcal{T}}=\{X^\alpha|\ 2\alpha\in
\cal{C}_{f,\mathcal{T}}\}$. 
%fguo add:
If $\mathcal{T}=\calT{X}{e}$, we 
write the shorthand $\mathcal{G}_{f,\deg\le e}$.
\end{define}
By Theorem \ref{th:sparsity}, we have
$m_{\cal{G}_{f,\mathcal{T}}}\subseteq m_{\calT{X}{d}}$ and that
$f \in \RSOS{\mathcal{T}}$ if and only if
\begin{equation*}
0={m_{\cal{G}_{f,\mathcal{T}}}}^TW^{[1]}m_{\cal{G}_{f,\mathcal{T}}}+
(-f(X))\cdot {m_{\mathcal{T}}}^T W^{[2]} m_{\mathcal{T}}.
\end{equation*}
Thus the sizes of the SDPs (\ref{primal}) and (\ref{dual})
decrease. We show a sparse version of Theorem \ref{th::main2} below.%lzhi was: The Corollary \ref{th::main2}

\begin{cor}\label{LinFormCor2}
Given a polynomial $f\in \QQ[X]$ and an integer $e\ge 0$, let
 $d=\lceil e+\deg(f)/2\rceil$,
  then for any subset $\mathcal{T}\subseteq \calT{X}{e}$,
the following are equivalent:
%lzhi was:
%Given polynomial $f\in\QQ[X]$ and integer $e\ge 0$, let
%$\RR[X]_{2d}:=\{p\in\RR[X]\ |\ \text{supp}(p)\in \calT{X}{2d}\}$ where
%$d=\lceil e+\deg{(f)}/2\rceil$, then for any subset $S[X;e]\subseteq
%\calT{X}{e}$, the following are equivalent:
\begin{enumerate}[\upshape 1.]
\item $f\notin\RSOS{\mathcal{T}}$, 
\item There exists a rational vector $\hat{y}\in \QQ^{m}$
and the associated linear form $L_{\hat{y}}\in {(\RR[X]_{2d})}^*$ 
such that for any polynomials $v, u\in\RR[X]$ with $\text{supp}(v)\in \mathcal{T}$
and $\text{supp}(u)\in \mathcal{G}_{f,\mathcal{T}}$, we have
$L_{\hat{y}}(fv^2)< 0$ and $L_{\hat{y}}(u^2)\ge 0$, %lzhi was:.
\end{enumerate}
where $m$ is the number of  elements in the set
$\{X^{\alpha+\beta}\ |\ X^{\alpha}, X^{\beta} \in
\mathcal{G}_{f,\mathcal{T}}\}$.
\end{cor}

\section{Hilbert-Artin representation of positive semidefinite rational
functions}\label{sec::rafuns}

We generalize our method for solving %lzhi was: to
 the following problem: Given a rational function
$f/g\in\QQ(X)$ with $g(X)\succeq 0$ %fguo was: $g(X)\ge 0$
an integer $e\ge 0$ and
$\mathcal{T}\subseteq \calT{X}{d}$, %lzhi2 was: \mathcal{T}[X;e]$, %lzhi was: how to
 certify $f/g\notin
\RSOS{\mathcal{T}}$.

Consider the following set
\begin{equation}\label{optproblem2}
        \left\{ \begin{array}{c|c}[W^{[1]},W^{[2]}]&
        \begin{aligned}
         &\ g(X)\cdot{m_{\calT{X}{d}}}^TW^{[1]}m_{\calT{X}{d}}=f(X)
     \cdot {m_{\mathcal{T}}}^T W^{[2]} m_{\mathcal{T}}\\
     &\ W^{[1]} \succeq 0,\ W^{[2]}\succeq 0,\ \text{Tr}(W^{[2]})=1
        \end{aligned}
    \end{array}
    \right\},
\end{equation}
where $d=e+(\lceil\deg(f)-\deg(g)\rceil)/2$.

%lzhi add:
\begin{prop}
We have $f/g\notin \RSOS{\mathcal{T}}$ if and only if the set
{\upshape(\ref{optproblem2})} is empty.
\end{prop}

%lzhi rewrite:
  Let
\begin{equation}\label{TwoSets}
\begin{aligned}
&\Gamma_1:=\left\{X^{\alpha+\beta+\gamma}\ |\ X^{\gamma}\in \text{supp}(g), X^{\alpha},X^{\beta}\in\calT{X}{d}\right\},\\
&\Gamma_2:=\left\{X^{\alpha+\beta+\gamma}\ |\ X^{\gamma}\in \text{supp}(f), X^{\alpha},X^{\beta}\in \mathcal{T}\right\}.
\end{aligned}
\end{equation}
We assume that $\Gamma_1\supseteq \Gamma_2$, otherwise $f/g\notin
\RSOS{\mathcal{T}}$. The following analysis is similar to the one given
in Section \ref{sec::polys}.
%lzhi was:  we can make an analogical analysis given as is
%done in Section \ref{sec::polys}. Different with (\ref{def::G}),

The  primal block SDP considered here  has the   same form as
(\ref{primal}) but we use
\begin{equation}\label{newG}
     g(X)\cdot{m_{\calT{X}{d}}}^T W^{[1]} m_{\calT{X}{d}}=
     \underset{\alpha}{\sum}(G^{[\alpha]}\bullet
     W^{[1]})X^{\alpha}
\end{equation}
to define   matrices $G^{[\alpha]}$.  %The closedness statement in Lemma
%\ref{lem::priclosed} is also true for the case in this section and we
%omit it and its proof.
Its  dual problem is
\begin{equation}\label{dual2}
        \begin{aligned}
             s^*:=\underset{(y,s)\in \RR^{m+1}}{\inf} & \ s\\
                %s.t.  &\ M(y,s)+I \succeq 0,\\
                s.t.  &\ M(y,s)\succeq 0,\\
        \end{aligned}
\end{equation}
where
\[
    M(y,s):=\left[\begin{array}{cc}M_d(gy) & \\ &
    M_e((-f)y)+sI\end{array}\right],
\]
 %lzhi was: $y:=(y_{\alpha})_{\alpha \in \NN_{2d}^n}\in \RR^{\NN_{2d}^n}$
%and $m=|\NN_{2d}^n|$.
$y:=(y_{\alpha})_{\alpha \in \NN_{2d}^n}\in \RR^{\NN_{2d}^n}$ and
%fguo was: $m={\tbinom{n+2d}{2d}}$,
$m$ is the number of elements in the set $\Gamma_1$. $M_d(gy)$ and
$M_e((-f)y)$ are localizing moment matrices. Similar to Lemma
\ref{lem::SF}, we have
\begin{lemma}\label{lem::SF::rafuns}
There exists $\tilde{p}=(\tilde{y},\tilde{s})$ such that $M_d(g\tilde{y})\succ 0$
and $M_e(-f\tilde{y})+\tilde{s}I\succ 0$.
\end{lemma}
\begin{proof}
%Since $g(X)$ is nonnegative and has finite real roots, we can obtain
%$(\tilde{y},\tilde{s})$ as is done in the proof of Lemma \ref{lem::SF}.
%let $\mu$ be a probability measure on $\RR^n$ with a {\itshape strictly positive}
%density $h$ with respect to Lebesgue measure such that
%\[
%        \tilde{y}_{\alpha}:=\int X^{\alpha} d\mu < \infty.
%\]
%For any polynomial $q(X)\in\RR[X]$ with $\text{supp}(q)\in \calT{X}{d}$, let
%$\text{vec}(q)$ denote its sequence of coefficients in the monomial basis
%$\calT{X}{d}$.
Taking $\tilde{y}_{\alpha}$ to be the one defined in the proof of Lemma
\ref{lem::SF}, since $g(X)$ is nonnegative, for any polynomial $q(X)\in\RR[X]$
with $\text{supp}(q)\in \calT{X}{d}$, we have
\begin{equation*}
\begin{aligned}
\langle \text{vec}(q)^T,M_d(g\tilde{y})\text{vec}(q)\rangle&=\int g(x)q(x)^2\mu(dx)\\
&=\int g(x)q(x)^2h(x)dx\\
&>0\ \text{whenever}\ q\neq 0, %lzhi was:.
\end{aligned}
\end{equation*}
which implies $M_d(g\tilde{y})\succ 0$. %Then we choose $\tilde{y}=\frac{1}{2}y$ and
Take
\[
\tilde{s}>-\lambda_{\text{min}}(M_e((-f)\tilde{y})),
\]
then $M_e(-f\tilde{y})+\tilde{s}I\succ 0$.
\end{proof}

Based on the Farkas %ek4 add
Lemma~\ref{lem::Farkas}
and Lemma \ref{lem::SF::rafuns},
 similar to Theorem \ref{th::main1} and Theorem \ref{th::main2}, we have the following
 results.
\begin{theorem}\label{th::main1::rafuns}
Given a rational function $f/g\in \QQ(X)$ with $g(X)\succeq 0$ and an
integer $e\ge 0$, let $d=e+(\lceil\deg(f)-\deg(g)\rceil)/2$, then
for any subset $\mathcal{T}\subseteq \calT{X}{e}$, the following are
equivalent:
\begin{enumerate}[\upshape 1.]
\item $f/g\notin \RSOS{\mathcal{T}}$,
\item $\Gamma_1\nsupseteq\Gamma_2$ in {\upshape(\ref{TwoSets})}, or
      there exists a rational vector $\hat{y}=(\hat{y}_\alpha)\in \QQ^{m}$
      such that $M_d(g\hat{y})\succeq 0$ and $M_e(f\hat{y})\prec 0$
      in {\upshape(\ref{dual2})},
\end{enumerate}
where $m$ is the number of elements in the set $\Gamma_1$.
\end{theorem}
With a view towards linear forms in $\RR[X]$, Theorem \ref{th::main1::rafuns}
is equivalent to
\begin{theorem}\label{th::main2::rafuns}
Given a rational function $f/g\in \QQ(X)$ with $g(X)\succeq 0$ and an
integer $e\ge 0$, let $\RR[X]_{2d+\deg(g)}:=\{p\in\RR[X]\ |\
\text{supp}(p)\in \calT{X}{2d+\deg(g)}\}$ where
$d=e+(\lceil\deg(f)-\deg(g)\rceil)/2$, then for any subset
$\mathcal{T}\subseteq \calT{X}{e}$, the following are equivalent:
\begin{enumerate}[\upshape 1.]
\item $f/g\notin\RSOS{\mathcal{T}}$,
\item $\Gamma_1\nsupseteq\Gamma_2$ in {\upshape(\ref{TwoSets})}, or
there exists a rational vector $\hat{y}\in \QQ^{m}$ and the associated linear form
$L_{\hat{y}}\in {(\RR[X]_{2d+\deg(g)}})^*$, such that for any polynomials
$v, u\in\RR[X]$ with $\text{supp}(v)\in \mathcal{T}$ and $\text{supp}(u)\in \calT{X}{d}$,
we have $L_{\hat{y}}(fv^2)< 0$ and $L_{\hat{y}}(gu^2)\ge 0$,
\end{enumerate}
where $m$ is the number of elements in the set $\Gamma_1$.
\end{theorem}

\section{Examples and Experiments}

%lzhi rewrite
\begin{example}\label{ex::Motzkin}
We prove that the well known Motzkin polynomial
\[
f(X_1,X_2)=X_1^4X_2^2+X_1^2X_2^4+1-3X_1^2X_2^2
\]
is not an SOS. %lzhi add: an
 Set $n=2$, $e=0$ and $d=3$. By exploiting the
sparsity, we have %Fguo was $\mathcal{G}_{f,S[X;0]}=\{1,X_1X_2,X_1^2X_2,X_1X_2^2\}$. %lzhi was: [2;0]
$\mathcal{G}_{f,\deg\le 0}=\{1,X_1X_2,X_1^2X_2,X_1X_2^2\}$.
%lzhi delete:
%which means if $f$ has an SOS decomposition then $f$ can be
%written $f=\sum u_i(X,Y)^2$, $\text{supp}(u_i)\subseteq
%\mathcal{G}_{f,S[X;0]}$. %lzhi was: \mathcal{G}_{f,[2;0]}$
According to Corollary \ref{LinFormCor2}, $m=10$ and we need to
find a rational sequence $\hat{y}\in \QQ^{10}$ or its associated
linear form $L_{\hat{y}}\in {(\RR[X]_{6})}^*$ such that for any polynomial
$u\in\RR[X]$
with %fguo was: $\text{supp}(u)\in \mathcal{G}_{f,S[X;0]}$, %lzhi was: \cal{T}[2;0]
$\text{supp}(u)\in \mathcal{G}_{f,\deg\le 0}$,
we have $L_{\hat{y}}(u^2)\ge 0$ and $L_{\hat{y}}(f)< 0$. The certificate we obtained %using {\itshape SDPTools\/}
is
%\begin{equation*}
%\begin{aligned}
%\hat{y}=&(\hat{y}_{0,0}=\frac{22011}{55402},\hat{y}_{1,1}=0,
%\hat{y}_{2,1}=0,\hat{y}_{1,2}=0,\hat{y}_{2,2}=\frac{358944}{9403},
%\hat{y}_{3,2}=0, \hat{y}_{2,3}=0, \hat{y}_{4,2}=\frac{96310}{4693},\\
%&\hat{y}_{3,3}=0,\hat{y}_{2,4}=\frac{96310}{4693}).
%\end{aligned}
%\end{equation*}
\begin{equation*}
\begin{aligned}
\hat{y}=&(\hat{y}_{0,0}=\hat{y}_{1,1}=\hat{y}_{1,2}=0,\hat{y}_{2,2}=300,
\hat{y}_{3,2}=\hat{y}_{2,3}=\hat{y}_{4,2}=\hat{y}_{3,3}=\hat{y}_{2,4}=0).
\end{aligned}
\end{equation*}
%Let $f={\sum_{\alpha}}f_{\alpha}X^{\alpha_1}Y^{\alpha_2}$ and
%$u={\sum_{\alpha}}u_{\alpha}X^{\alpha_1}Y^{\alpha_2}$, then
Its associated linear form $L_{\hat{y}}$ satisfies %$$L_{\hat{y}}(u^2)=\frac{22011}{55402}u_{0,0}^2+\frac{358944}{9403}u_{1,1}^2+
%\frac{96310}{4693}u_{2,1}^2+\frac{96310}{4693}u_{1,2}^2\ge 0$$ and
%$$L_{\hat{y}}(f)=\frac{96310}{4693}+\frac{96310}{4693}+\frac{22011}{55402}-
%3\times \frac{358944}{9403}=-\frac{178662293250763}{2444794913158}<0,$$
$$L_{\hat{y}}(u^2)=300u_{1,1}^2\ge 0,$$ and
$$L_{\hat{y}}(f)=-3\times 300=-900<0,$$
which implies $f$ can not be written as an SOS.
\end{example}

%lzhi rewrite
\begin{example}
In \citeP{KYZ09}, the %EK authors prove the
monotone column permanent (MCP) conjecture has been  proven %lzhi was: is
for
dimension 4 via certifying  polynomials  $p_{1,1}, p_{1,2}$,
$p_{1,3}, p_{2,2}, p_{2,3}, p_{3,3}$ of degree 8  in   8 variables
to be \nonnegative,
see \citeP{KYZ09} for the explicit forms of
these polynomials.  Among them, the polynomials $p_{1,1}, p_{3,3}$
are perfect squares. Applying the hybrid symbolic-numeric
algorithm in \citeP{KLYZ09}, they proved that the polynomial
$p_{1,3}$ can be written as an SOS  and the polynomials
$p_{1,2},p_{2,2}, p_{2,3}$ can be written as an SOS divided by
weighted sums of squares of variables. We certify that all
polynomials $p_{1,2},p_{2,2}, p_{2,3}$ can not be written as an
SOS via finding the corresponding certificates $\hat{y}\in \QQ^m$ and the
associated linear forms $L_{\hat{y}}\in {(\RR[X]_{8})}^*$.
 By exploiting the sparsity,
for $p_{1,2}$, the matrices $W$, $M(y,s)$ in (\ref{primal}),
(\ref{dual}) are of dimension $24\times 24$ and $m=189$. For
$p_{2,2}$, $W$ and $M(y,s)$ are of dimension $29\times 29$ and
$m=255$.   For $p_{2,3}$, $W$ and $M(y,s)$ are of dimension
$39\times 39$ and $m=372$.
\end{example}

\begin{example}\label{ex:ess}
This example comes from the even symmetric sextics in \citeP{ESS}. Let
\begin{equation*}
M\ess{n}{r}(X)=\overset{n}{\underset{i=1}{\sum}}X_i^r,
\end{equation*}
 for integers $k$, $0\le k\le n-1$, we define polynomials  $f_{n,k}$ by
\begin{equation*}
f\ess{n}{0}=-nM\ess{n}{6}+(n+1)M\ess{n}{2}M\ess{n}{4}-M\ess{n}{2}^3,
\end{equation*}
and
\begin{equation*}
f\ess{n}{k}=(k^2+k)M\ess{n}{6}-(2k+1)M\ess{n}{2}M\ess{n}{4}+M\ess{n}{2}^3,\ 1\le k\le n-1.
\end{equation*}
Some interesting results about these polynomials have been given
in \citeP{ESS}.
\begin{prop}\label{prop::ESS}
For $n\ge 3$,
\begin{enumerate}[{\upshape (1)}]
\item all $f\ess{n}{k}$, $0\le k\le n-1$, are \nonnegative
polynomials;
\item the polynomials  $f\ess{n}{0}$ and
$f\ess{n}{1}$ are SOS; \item the polynomials
$f\ess{n}{2},\ldots,f\ess{n}{n-1}$ are not SOS; \item
$M\ess{3}{2}\cdot f\ess{3}{2}$ is an SOS\/ {\upshape\citep{SPDNSOSRobinson};}
      $M\ess{4}{2}^2 \cdot f\ess{4}{2}$ is an SOS;
\item for $n\ge 4$, $M\ess{n}{2}\cdot f\ess{n}{n-1}$ is an  SOS.
\end{enumerate}
\end{prop}
For $n\ge 4$ and $2\le i\le n-2$,  we wish to know whether
$M\ess{n}{2}\cdot f\ess{n}{i}$ is an  SOS. We have the following
results.
\begin{enumerate}[Ex.\ref{ex:ess}.1]%ek2

\item For $n=4$, we can certify that the polynomial
\fbox{$f\ess{4}{2} \notin \RSOSdeg{2}$.} %ek2
By exploiting the
sparsity, in (\ref{primal}), $W^{[1]}$ has dimension $55\times 55$
and $W^{[2]}$ has dimension $5\times 5$. We have $m=369$ in
(\ref{dual}).

\item For $n=5$, we can certify the following:
\begin{equation*}
\fbox{$\displaystyle f\ess{5}{2}\notin \RSOSdeg{2}\text{ and }f\ess{5}{3} \notin \RSOSdeg{4}.$}
\end{equation*}
By exploiting the
sparsity, for $f\ess{5}{2}$, $W^{[1]}, W^{[2]}$ have dimension $105\times
105, 6\times 6$, respectively and $m=1036$. For $f\ess{5}{3}$,
$W^{[1]}, W^{[2]}$ have dimension $231\times
231, 21\times 21$, respectively and $m=2751$.

\item For $n=6$, we can certify the following:
\begin{equation*}
\fbox{$\displaystyle f\ess{6}{2}\notin \RSOSdeg{2}\text{ and }f\ess{6}{3},
f\ess{6}{4} \notin \RSOSdeg{4}.$}
\end{equation*}
By exploiting the
sparsity, for $f\ess{6}{2}$, $W^{[1]}, W^{[2]}$ have dimension $182\times
182, 7\times 7$, respectively and $m=2541$. For $f\ess{6}{3}$ and $f\ess{6}{4}$,
$W^{[1]}, W^{[2]}$ have dimension $434\times
434, 28\times 28$, respectively and $m=7546$.
\end{enumerate}

%ek2 To our knowledge, polynomials $f\ess{4}{2}$ and
%ek2 $f\ess{5}{2},f\ess{5}{3}$ are first set of examples which are
%ek2 \nonnegative but can not be written as a ratio of sums of squares
%ek2 with the degree of denominator being less than 4.
\end{example}

\begin{example}\label{ex::illposed}
Consider the polynomial
$f(X_1,X_2)=X_1^2+X_2^2-2X_1X_2=(X_1-X_2)^2$. %lzhi delete: \in\RR[X_1,X_2]$.
Its minimum is $0$.  %Since
%$f=(X_1-X_2)^2$,  we have $\inf f=0$.
%$\underset{x_1,x_2\in\RR}{\inf}{f(x_1,x_2)}=0$.
 However, for any small perturbation %ek4 perturtation
$\varepsilon>0$, the  polynomial
$f_{\varepsilon}(X_1,X_2)=(1-{\varepsilon}^2)X_1^2+X_2^2-2X_1X_2$
is not an SOS. Indeed, %taking $x_1=x_2=C\in\RR$,
$f_{\varepsilon}(C,C)=-{\varepsilon}^2C^2$ which implies that the
 infimum of $f_{\varepsilon}$ is  $-\infty$.
%${\inf}{f_{\varepsilon}}=-\infty$.
%$\underset{x_1,x_2\in\RR}{\inf}{f_{\varepsilon}(x_1,x_2)}=-\infty$.
Hence $f$ is an {\itshape ill-posed} polynomial \citeP{HKZ10}. For
$\varepsilon=10^{-1},\ldots,10^{-5}$, we can use Matlab SDP solver
{SeDuMi} in Step \ref{getiter} in Algorithm \ref{alg} to certify
that $f_{\varepsilon}$ is not SOS.
%can numerically detect $f_{\varepsilon}$ is not SOS.
But for $\varepsilon<10^{-5}$, %it fails and gives wrong SOS decomposition.
Step \ref{getiter} does not work out and we are not able to obtain
a rational  solution at which $s_k<0$. If we use the command
$\mathsf{findsos}$ in SOSTOOLS \citeP{sostools}, it outputs a
wrong SOS decomposition.  Our method implemented in {\itshape
SDPTools} in Maple can give exact certificates for
$f_{\varepsilon}$ being not an SOS for $\varepsilon=10^{-8}$ or
smaller! %lzhi: ??? please give the certificate for $\varepsilon=10^{-8}$
Take $\varepsilon=10^{-8}$ for instance. By exploiting the
sparsity, we have %fguo was: $\mathcal{G}_{f_{\varepsilon},S[X;0]}=\{X_1,X_2\}$.
$\mathcal{G}_{f_{\varepsilon},\deg\le 0}=\{X_1,X_2\}$.
Setting $Digits=45$ in Maple, the certificate we obtained is
\begin{equation*}
\begin{aligned}
&\hat{y}=\left(\hat{y}_{2,0}=\frac{46635362642387337096986}{1731626131338905851065},
\hat{y}_{1,1}=\frac{53470001073377890290267}{1985404333861113854675},\right. \\
&\quad\quad\left.
\hat{y}_{0,2}=\frac{19926414238854847715525}{739891310902398542446}\right).
 %lzhi2 was: .\right)
\end{aligned}
\end{equation*}
For any $u\in\RR[X]$ with %fguo was: $\text{supp}(u)\in \mathcal{G}_{f_{\varepsilon},S[X;0]}$,
$\text{supp}(u)\in \mathcal{G}_{f_{\varepsilon},\deg\le 0}$,
we have
\begin{equation*}
\begin{aligned}
L_{\hat{y}}(u^2)&=\frac{46635362642387337096986}{1731626131338905851065}u_{1,0}^2
+\frac{19926414238854847715525}{739891310902398542446}u_{0,1}^2\\
&\quad +2\times \frac{53470001073377890290267}{1985404333861113854675}u_{1,0}u_{0,1}\\
&\ge |2u_{1,0}u_{0,1}|\left(\left(\frac{46635362642387337096986}{1731626131338905851065}
\times \frac{19926414238854847715525}{739891310902398542446}\right)^{\frac{1}{2}}\right. \\
&\quad\left.
-\frac{53470001073377890290267}{1985404333861113854675}\right)\ge 0. %lzhi2??? you may show the exact number.
\end{aligned}
\end{equation*}
However,
\begin{equation*}
\begin{aligned}
L_{\hat{y}}(f_{\varepsilon})=&\frac{9999999999999999}{10000000000000000}\times
\frac{46635362642387337096986}{1731626131338905851065}+
\frac{19926414238854847715525}{739891310902398542446}\\
&-2\times\frac{53470001073377890290267}{1985404333861113854675}<0, %lzhi2??? you may show the exact number.
\end{aligned}
\end{equation*}
which implies $f_{\varepsilon}$ is not SOS.
\end{example}

\begin{example}\label{ex::rf}
In this example, we consider some rational functions.% which are constructed according to
%the following proposition.
\begin{enumerate}[Ex.\ref{ex::rf}.1]
\item For Motzkin polynomial in Example \ref{ex::Motzkin}, %setting $g(X_1,X_2)=X_1^2+1$,
we can certify that
\[
\frac{X_1^4X_2^2+X_1^2X_2^4+1-3X_1^2X_2^2}{X_1^2+1}\notin \RSOSdeg{2}.
\]
\item %In view of Proposition \ref{prop::ESS},
%$f\ess{n}{2},\ldots,f\ess{n}{n-1}$ are not SOS for $n\ge 3$.
%Setting $g_n(X)=M\ess{n}{2}$,
For the even symmetric sextics in Example \ref{ex:ess},
we can certify that
\[
\frac{f\ess{n}{2}}{M\ess{n}{2}}, \ldots,
\frac{f\ess{n}{n-1}}{M\ess{n}{2}}\notin \RSOSdeg{2}, \quad n=4,5,6.
\]
Those constitute our largest certificates. %ek3 add  Feng, how large?
\end{enumerate}
The correctness of the above result is guaranteed by the following proposition.

%ek3 \begin{prop}\label{prop::rf}
%ek3 Given a rational function $f/g\in\RR(X)$ with $g(X)\succeq 0$,
%ek3 suppose that $g(X)$ is irreducible over $\RR$ and $g\nmid f$, then $f\notin$ SOS
%ek3 implies $f/g\notin\RSOSdeg{\deg(g)}$.
%ek3 \end{prop}
%ek3 \begin{proof}
%ek3 Assume that
%ek3 $$\frac{f}{g}=\frac{\sum_{i}^{l}u_i(X)^2}{\sum_{j}^{l'}v_j(X)^2}$$
%ek3 for some polynomials $u_i(X),v_j(X)\in\RR[X]$ with $\deg(v_j)\le
%ek3 \deg(g)/2$. By the assumption, we have $g\mid \sum_j^{l'}v_j^2$
%ek3 which means $g(X)=c\sum_j^{l'}v_j(X)^2$
%ek3 %lzhi2 was:$g\mid \sum_j^{l'}v_j$, $g(X)=c\sum_j^{l'}v_j(X)$
%ek3 for some constant $c\in\RR$. It contradicts that $f\notin$ SOS.
%ek3 \end{proof}

%ek3 new prop
\begin{prop}\label{prop::rf}
Let $f/g\in \RR(X)$ be a multivariate rational function
where $f,g\in \RR[X]$ with $\text{\upshape GCD}(f,g)=1$.
If $f,-f\notin \text{\upshape SOS}$ then $f/g\notin\RSOSdeg{\deg(g)}$.
\end{prop}
\begin{proof}
Assume the contrary, namely that
\begin{equation}\label{eq:sosfrac}
   \frac{f}{g}=\frac{\sum_{i=1}^{l}u_i(X)^2}{\sum_{j=1}^{l'}v_j(X)^2},
   \quad u_i(X),v_j(X)\in\RR[X],\quad \deg(v_j)\le \deg(g)/2.
\end{equation}
Thus the right-side of (\ref{eq:sosfrac}) constitutes the
reduced fraction $f/g$,
which means $g(X)=c\sum_j^{l'}v_j(X)^2$
%lzhi2 was:$g\mid \sum_j^{l'}v_j$, $g(X)=c\sum_j^{l'}v_j(X)$
for some non-zero constant $c\in\RR$, making $f/c = \sum_{i=1}^{l}u_i(X)^2$,
a contradiction. %ek4 contradition.
\end{proof}

Furthermore, for the polynomials in Example \ref{ex:ess},
we can compute the certificates for the following result:
\begin{equation}\label{eq:choirf}
\frac{f\ess{n}{2}}{M\ess{n}{2}}\notin\RSOSdeg{4}, \quad n=4, 5, 6
\quad \text{and} \quad
\frac{f\ess{5}{3}}{M\ess{5}{2}}\notin\RSOSdeg{6}.
\end{equation}
We have no generalization of Proposition~\ref{prop::rf} to $\pm f\notin\RSOSdeg{2e}$,
and the impossibilities (\ref{eq:choirf}) may hint of new unknown properties
of the even symmetric sextics in \citeP{ESS}.
\end{example}

 \section*{Acknowledgments}
We thank Bruce Reznick for kindly providing us the even symmetric sextics
in Example \ref{ex:ess}.

\fi%skiptext

\def\refname{\Large\bfseries References}

\bibliographystyle{myplainnat}
\bibliography{%
strings%
,kaltofen% Erich's bib files
,fguo% Feng's bib files
,certifiNotRatSOS% new citations
,crossrefs%
}

\end{document}